\documentclass[12pt]{article}

\usepackage{amsmath}
\usepackage{amssymb}
\usepackage{amscd}
\usepackage{amsthm}
\usepackage{indentfirst}

\usepackage[mathscr]{euscript}

\newcommand{\Z}{\ensuremath{\mathbb{Z}}}
\newcommand{\Q}{\ensuremath{\mathbb{Q}}}

\theoremstyle{plain}
\newtheorem{thm}{Theorem}[section]
\newtheorem{lem}[thm]{Lemma}
\newtheorem{cor}[thm]{Corollary}

\newtheorem{propdfn}[thm]{Proposition-Definition}
\newtheorem{question}[thm]{Question}
\theoremstyle{definition}

\newtheorem{rmq}[thm]{Remark}

\DeclareMathOperator{\Spec}{Spec}
\DeclareMathOperator{\Pic}{Pic}
\DeclareMathOperator{\Jac}{Jac}
\DeclareMathOperator{\Cl}{Cl}
\DeclareMathOperator{\ind}{ind}

\DeclareMathOperator{\rk}{rk}
\DeclareMathOperator{\img}{Im}

\newcommand{\gm}{\mathbf{G}_{{\rm m}}}
\newcommand{\V}{\mathcal{V}}
\newcommand{\X}{\mathcal{X}}
\newcommand{\J}{\mathcal{J}}
\newcommand{\C}{\mathscr{C}}

\newcommand{\tors}{\text{\rm tors}}

\begin{document}

\title{Pulling back torsion line bundles to ideal classes}

\author{Jean Gillibert \and Aaron Levin}


\maketitle

\begin{abstract}
We prove results concerning the specialisation of torsion line bundles on a variety $V$ defined over $\Q$ to ideal classes of number fields. This gives a new general technique for constructing and counting number fields with large class group.
\end{abstract}


\section{Introduction}

The purpose of this paper is to study the specialisation of torsion line bundles, on a variety $V$ defined over $\Q$, to ideal classes of number fields, and to make this operation optimal, that is, to make the kernel of specialisation as small as possible.

The main motivation for introducing this technique is the construction of number fields whose class group has large torsion from varieties (over $\Q$) whose Picard group has large torsion. For instance, we show how the problem of constructing quadratic number fields with large class group can be reduced to the problem of finding a hyperelliptic curve with a rational Weierstrass point and a large rational torsion subgroup in its Jacobian (see Corollary~\ref{chyp}).  This technique is an abstraction and generalisation of the method used in \cite{Levin}, which used certain superelliptic curves to obtain new constructions of number fields with a large ideal class group.  Further results on the construction of large ideal class groups will appear in a separate paper by the second author.  We mention also the closely related geometric techniques used by Mestre in \cite{Mes,Mes3,Mes2,Mes4} to construct number fields with a large ideal class group.

On the other hand, the question of pulling back line bundles on arithmetic varieties to nontrivial ideal classes of number fields has been raised by Agboola and Pappas in \cite{bisipappas}. Our results positively answer their question in the case of torsion line bundles on a hyperelliptic curve (see Corollary~\ref{partial_answer} for the precise statement).

A consequence of our main result is the following:

\begin{thm}
\label{intro_theorem}
Let $V$ be a smooth projective variety over $\Q$, and let $m>1$ be an integer. Let $S$ be the set of places of bad reduction of $V$, and let $\V\rightarrow \Spec(\Z_S)$ be a smooth projective model of $V$. Then there exists an infinity of number fields $K$ with a point $P\in V(K)$ such that the specialisation map
$$
P^*:\Pic(V)[m]=\Pic(\V)[m]\longrightarrow \Cl(\mathcal{O}_{K,S})[m]
$$
satisfies
$$
\rk_m \img(P^*)\geq \rk_m \Pic(V)[m] +\#S - \rk \mathcal{O}_{K,S}^{\times}.
$$
More precisely, given a generically finite rational map $V\rightarrow \mathbb{A}^l$ of degree $d$, it is possible to find an infinity of such $K$ of degree $d$.
\end{thm}

Two short comments: 1) in our terminology, a variety over $\Q$ is a geometrically integral separated scheme of finite type over $\Q$. 2) If $A$ is a finitely generated abelian group and $m\geq 2$ an integer, $\rk_m A$ denotes the largest integer $r$ such that $A$ has a subgroup isomorphic to $(\Z/m)^r$. See \S\ref{notation_def} for further notation.

Of course, results analogous to those given here can also be established when the variety $V$ is defined over any number field instead of $\Q$.
Another generalisation would be to replace line bundles, i.e. $\gm$-torsors, by torsors under other group schemes, like tori.

Let us briefly describe the proof of Theorem~\ref{intro_theorem}. We first choose $r$ linearly independent elements of order $m$ in the group $\Pic(V)=\Pic(\V)$ (see Lemma~\ref{good_model}). By Kummer theory over $\V$, it is possible to lift these $r$ elements to $r$ elements in the flat (fppf) cohomology group $H^1(\V,\mu_m)$. These classes in $H^1(\V,\mu_m)$ are represented by $\mu_m$-torsors $\X_i\rightarrow \V$. Let $X_i\rightarrow V$ be the restriction of $\X_i$ to $V$, and let $X\rightarrow V$ be the fiber product over $V$ of the $X_i$. We choose a generically finite rational map $V\rightarrow \mathbb{A}^l$ of degree $d$ and apply Hilbert's irreducibility theorem to the composite cover $X\rightarrow V\rightarrow \mathbb{A}^l$. Thus, we get an infinity of points in $\mathbb{A}^l$ whose inverse image in $V$ is a point $P\in V(K)$, with $[K:\Q]=d$. Moreover, the images of the $\X_i$ by the map
$$
P^*:H^1(\V,\mu_m)\longrightarrow H^1(\mathcal{O}_{K,S},\mu_m)
$$
are $r$ linearly independent elements of order $m$. Now, we look at the image of these elements by the natural map $c:H^1(\mathcal{O}_{K,S},\mu_m)\rightarrow \Cl(\mathcal{O}_{K,S})$. By Kummer theory (see \S\ref{notation_def}), the kernel of this map is $\mathcal{O}_{K,S}^{\times}/m$, but we can improve the situation by asking that the subgroup generated by the $P^*\X_i$ has trivial intersection with the subgroup $\Z_S^{\times}/m$. It is also possible to ensure that $K$ is linearly disjoint from the cyclotomic field $\Q(\zeta_{2m})$. Under these conditions, the kernel of the map $c\circ P^*$ is isomorphic to a subgroup of $(\Z/m)^{\rk \mathcal{O}_{K,S}^{\times}-\# S}$, which proves the result after invoking Lemma~\ref{m_rank_sequence}.

The plan of this paper is as follows. In section~2, after recalling effective versions of Hilbert's irreducibility theorem, we prove our main result, Theorem~\ref{MainTheorem}, involving a variety $V$ that satisfies some special properties with respect to some set $S$ of prime numbers. Then, using the theory of integral models, we give consequences of our result in the case when $V$ is smooth and $S$ is the set of primes of bad reduction of $V$. In section~3, using superelliptic and hyperelliptic curves, we give applications of our technique to the construction of number fields with large class group. Finally, we explain how these results can be applied, in the case of torsion line bundles on hyperelliptic curves, to the question of Agboola and Pappas.

\subsection*{Acknowledgements}

We thank Qing Liu for inspiring conversations, and C\'edric Pepin for technical advice about line bundles on integral models.


\section{Specialisation of $\mu_m$-torsors}

\subsection{Notation and definitions}
\label{notation_def}

By a variety over a field $k$, we will mean a geometrically integral separated scheme of finite type over $k$.

If $A$ is an abelian group, and if $m>1$ is an integer, we will denote by $A[m]$ the kernel of multiplication by $m$ in $A$ (that is, the $m$-torsion of $A$), and by $A/m$ the cokernel of multiplication by $m$ in $A$. Also, if $A$ is a finitely generated abelian group, the $m$-rank of $A$, which we denote by $\rk_m A$, is the largest integer $r$ such that $A$ has a subgroup isomorphic to $(\Z/m)^r$.

In the following, $S$ will denote a finite set of prime numbers. If $K$ is a number field, we will denote by $\mathcal{O}_{K,S}$ the ring of $S_K$-integers of $K$, where $S_K$ is the set of primes of $K$ lying above primes in $S$.
In the particular case when $K=\Q$, we denote this ring by $\Z_S$ instead of $\mathcal{O}_{\Q,S}$. Let us note that $\Z_S$ is nothing else than the ring of fractions $\Z[S^{-1}]$.

We recall that, by Kummer theory, we have a canonical isomorphism

$$
H^1(K,\mu_m)\simeq K^{\times}/m.
$$
Under this isomorphism, one identifies the first cohomology group $H^1(\mathcal{O}_{K,S},\mu_m)$, computed using the fppf topology, as the following subgroup of $H^1(K,\mu_m)$ :
$$
H^1(\mathcal{O}_{K,S},\mu_m)=\{z\in K^{\times}/m ~|~ \forall \mathfrak{p}\notin S_K, v_{\mathfrak{p}}(z)\equiv 0 \pmod{m}\},
$$
where $\mathfrak{p}$ runs through nonzero prime ideals of $K$. According to fppf Kummer theory over $\Spec(\mathcal{O}_{K,S})$, this group fits into the exact sequence
\begin{equation}
\label{kummer}
\begin{CD}
0 @>>> \mathcal{O}_{K,S}^{\times}/m @>>> H^1(\mathcal{O}_{K,S},\mu_m) @>>> \Cl(\mathcal{O}_{K,S})[m] @>>> 0.\\
\end{CD}
\end{equation}
Let us note that all the groups involved here are $m$-torsion.

Unless specified, all our cohomology groups are computed with respect to the fppf topology.

\subsection{Hilbert's Irreducibility Theorem}

An essential tool in our proofs is Hilbert's Irreducibility Theorem, which we now recall in a suitably general, quantitative form due to Cohen \cite{Coh2} (see also \cite[Ch. 9]{Ser}).

For a point $t=(t_1,\ldots, t_n)\in \mathbb{A}^n(\Z)$, define the height $H(t)=\max_j |t_j|$.
\begin{thm}[Hilbert Irreducibility Theorem]
\label{HIT}
Let $V$ be a variety over $\Q$ of dimension $l$.  Let $\phi:V\to \mathbb{A}^l$ be a generically finite rational map.  For $t\in \phi(V)(\Q)$, let $P_t=\phi^{-1}(t)$.  Let $k$ be a number field.  Then for all but $O\left(N^{l-\frac{1}{2}}\log N\right)$ points $t\in \mathbb{A}^l(\Z)$ with $H(t)\leq N$, we have $t\in \phi(V)$ and $P_t=\Spec(\Q(P_t))$, where $[k(P_t):k]=\deg\phi$.   If  $l=1$, $O\left(\sqrt{N}\log N\right)$ can be replaced by $O\left(\sqrt{N}\right)$ in the above.
\end{thm}

A very natural problem that arises in connection with Hilbert's Irreducibility Theorem is to determine the number of distinct number fields $\Q(P_{t})$ with $H(t)\leq N$ in Theorem~\ref{HIT}.  For $l=1$ this problem has been studied by Dvornicich and Zannier.  Let $C$ be a curve over $\Q$ and let $\phi:C\to \mathbb{P}^1$ be a morphism.  For each integer $t\in \mathbb{A}^1(\Z)\subset \mathbb{P}^1(\Q)$, let $P_t=\phi^{-1}(t)$.  Dvornicich and Zannier \cite{Zan, Zan2} studied the degree of the field extension $\Q(P_1,\ldots,P_N)$.  Their results imply in particular a useful result on the number of isomorphism classes of number fields in the set $\{\Q(P_1),\ldots,\Q(P_N)\}$.
\begin{thm}[Dvornicich, Zannier]
\label{Zan}
Let $C$ be a curve over $\Q$.  Let $\phi:C\to\mathbb{P}^1$ be a morphism with $\deg \phi>1$.  For each integer $t$, let $P_t=\phi^{-1}(t)$.  Let $g(N)$ denote the number of isomorphism classes of number fields in the set $\{\Q(P_1),\ldots,\Q(P_N)\}$.  Then $g(N)\gg \frac{N}{\log N}$.
\end{thm}

\begin{rmq}
The implied constants in Theorems~\ref{HIT} and \ref{Zan} are effective and can be explicitly computed.
\end{rmq}

\subsection{The main result}

Let $V$ be a variety over $\Q$. By Kummer theory, we have an exact sequence
$$
\begin{CD}
0 @>>> \gm(V)/m @>>> H^1(V,\mu_m) @>>> \Pic(V)[m] @>>> 0.\\
\end{CD}
$$
Given $L\in \Pic(V)[m]$, we say that a $\mu_m$-torsor on $V$ is a lift of $L$ if its image by the right-hand side map is equal to $L$.

\begin{thm}
\label{MainTheorem}
Let $V$ be a variety of dimension $l$ over $\Q$, and let $m>1$ be an integer. Assume that there exists a finite set $S$ of prime numbers such that, for any $L\in \Pic(V)[m]$, there exists a $\mu_m$-torsor $T\rightarrow V$ lifting $L$ such that, for any number field $K$ and for any point $P\in V(K)$, the pull-back $P^*T$ belongs to $H^1(\mathcal{O}_{K,S},\mu_m)$.

Let $\phi:V\rightarrow \mathbb{A}^l$ be a generically finite rational map. For $t\in \phi(V)(\Q)$, let $P_t:=\phi^{-1}(t)$.  Then for all but $O\left(N^{l-\frac{1}{2}}\log N\right)$ ($O(\sqrt{N})$ if $l=1$) points $t\in \mathbb{A}^l(\Z)$ with $H(t)\leq N$, we have $t\in \phi(V)$, $P_t=\Spec(\Q(P_t))$ where $\Q(P_t)$ is a number field of degree $deg(\phi)$, and
\begin{equation}
\label{MainTheoremEq}
\rk_m \Cl(\mathcal{O}_{\Q(P_t),S}) \geq \rk_m \Pic(V)_{\tors}+\# S- \rk \mathcal{O}_{\Q(P_t),S}^{\times}.
\end{equation}
\end{thm}

\begin{proof}
Let $r=\rk_m \Pic(V)_{\tors}$. Using the hypotheses of the theorem, we obtain $r$ torsors $X_i\rightarrow V$ that generate a subgroup isomorphic to $(\Z/m)^r$ in $H^1(V,\mu_m)$, and that satisfy the key property: for any number field $K$ and for any point $P\in V(K)$, the pull-back $P^*X_i$ belongs to $H^1(\mathcal{O}_{K,S},\mu_m)$.

Let $X$ be the product of the $X_i$ over $V$. Then $f:X\rightarrow V$ is irreducible of degree $m^r$, because the $X_i\rightarrow V$ are irreducible and linearly independent in $H^1(V,\mu_m)$. Let $k=\Q[\sqrt[m]{u}~|~u\in \Z_S^\times]$.  One can apply Hilbert's irreducibility theorem to the composition $X\rightarrow V\rightarrow \mathbb{A}^l$ : for all but $O\left(N^{l-\frac{1}{2}}\log N\right)$ ($O(\sqrt{N})$ if $l=1$) points $t\in \phi(V)\cap \mathbb{A}^l(\Z)$ with $H(t)\leq N$, we have $(\phi\circ f)^{-1}(t)=\Spec(F)$ where $F$ is a field and $[F:\Q]=[kF:k]=(\deg\phi)(\deg f)$. Now, if $t$ is such a point, then $\phi^{-1}(t)=P\in V(K)$ with $K\subseteq F$ and $[K:\Q]=\deg(\phi)$. Then, pulling back the torsors $X_i\rightarrow V$ along $P$, we get $r$ independent elements in $H^1(K,\mu_m)$ that, by the key property, belong to the subgroup $H^1(\mathcal{O}_{K,S},\mu_m)$. In other words,
$$
\rk_m H^1(\mathcal{O}_{K,S},\mu_m)\geq r.
$$
In fact, since $F$ is linearly disjoint from $k$, we have
$$
\rk_m H^1(\mathcal{O}_{K,S},\mu_m)/(\Z_S^\times/m)\geq r,
$$
where we identify $ \Z_S^\times/m$ with its image under the composite map $\Z_S^\times/m\to \mathcal{O}_{K,S}^{\times}/m \to H^1(\mathcal{O}_{K,S},\mu_m)$.  Note that $K$ is also linearly disjoint from $k$, and in particular is linearly disjoint from the cyclotomic field $\Q(\zeta_{2m})$, which is contained in $k$. So the map $\Z_S^\times/m\to \mathcal{O}_{K,S}^{\times}/m$ is injective and
$$
(\mathcal{O}_{K,S}^{\times}/m)/(\Z_S^\times/m)\simeq(\Z/m)^{\rk \mathcal{O}_{K,S}^{\times}-\#S}.
$$
On the other hand, starting from the Kummer exact sequence \eqref{kummer}, we deduce an exact sequence
$$
0 \longrightarrow (\mathcal{O}_{K,S}^{\times}/m)/( \Z_S^\times/m)\longrightarrow  H^1(\mathcal{O}_{K,S},\mu_m)/(\Z_S^\times/m)
\longrightarrow  \Cl(\mathcal{O}_{K,S})[m] \longrightarrow 0.
$$
This implies, by the second statement of Lemma~\ref{m_rank_sequence}, that
\begin{align*}
\rk_m\Cl(\mathcal{O}_{K,S})[m]&=\rk_m H^1(\mathcal{O}_{K,S},\mu_m)/(\Z_S^\times/m) - (\rk \mathcal{O}_{K,S}^{\times}-\#S) \\
&\geq r+\#S-\rk \mathcal{O}_{K,S}^{\times},
\end{align*}
hence the result.
\end{proof}

\begin{rmq}
\begin{enumerate}
\item[$(i)$] In the theorem, the set $S$ \emph{a priori} depends on $V$ and on $m$. In the applications we give,
$V$ is smooth and in this case one may take $S$ to be the set of places of bad reduction of $V$.
Moreover, in the case of curves (cf. Corollary~\ref{corollary2}), it is possible to find a smaller $S$ by taking $m$ into account.

\item[$(ii)$] The inequality \eqref{MainTheoremEq} also holds for the group $\Cl(\Q(P_t))$, because $\Cl(\mathcal{O}_{\Q(P_t),S})$ is a quotient of that group.
\end{enumerate}
\end{rmq}

\begin{lem}
\label{m_rank_sequence}
Let
\begin{equation}
\label{abc}
0 \longrightarrow A \longrightarrow B \longrightarrow C \longrightarrow 0
\end{equation}
be an exact sequence of finite $m$-torsion abelian groups. Then
$$
\rk_m B \geq \rk_m A + \rk_m C.
$$
Moreover, if $A$ or $C$ is isomorphic to $(\Z/m)^n$ for some $n$, then
$$
\rk_m B = \rk_m A +\rk_m C.
$$
\end{lem}

\begin{proof}
Let us first note that $m$-torsion abelian groups are modules over the ring $\Z/m$. It follows from the definition that, if $M$ is a finite $m$-torsion abelian group, and if $N$ is a subgroup or a quotient of $M$, then $\rk_m N\leq \rk_m M$.

We prove the last statement first. Let us assume that  $C$ is isomorphic to $(\Z/m)^n$ for some $n$. This means that $C$ is a free $(\Z/m)$-module of finite rank, thus is a projective object in the category of $(\Z/m)$-modules. Therefore, by basic homological algebra, the exact sequence \eqref{abc} splits, that is,
$$
B\simeq (\Z/m)^n\oplus A.
$$

Without loss of generality, we may assume that $m=p^e$ where $p$ is prime and $e\geq 1$. Using the structure theorem for finite abelian groups, we find that
$$
\rk_m M = \dim_{\mathbb{F}_p} p^{e-1}M
$$
for any $(\Z/m)$-module $M$. It follows that we have, for any two $(\Z/m)$-modules $M$ and $N$,
$$
\rk_m(M\oplus N)=\rk_m(M)+\rk_m(N).
$$
Finally, we get that $\rk_m B=n+\rk_m A$, as desired.

The case when $A$ is isomorphic to $(\Z/m)^n$ for some $n$ is similar to the previous one; indeed, it follows from the structure theorem for finite abelian groups that $(\Z/m)^n$ is an injective object in the category of $(\Z/m)$-modules, and so once again the exact sequence \eqref{abc} splits.

In the general case, let $n=\rk_m C$ and let us choose a subgroup of $C$ isomorphic to $(\Z/m)^n$ (we note that such a subgroup is not unique in general). Then, pulling-back the sequence \eqref{abc} by the inclusion $(\Z/m)^n\rightarrow C$, we get an exact sequence
$$
0 \longrightarrow A \longrightarrow B_1 \longrightarrow (\Z/m)^n \longrightarrow 0
$$
where $B_1$ is some subgroup of $B$. In particular, $\rk_m B\geq \rk_m B_1$. Now, by the second statement of the Lemma, we have $\rk_m B_1 = \rk_m A +n=\rk_m A+\rk_m C$, from which the result follows.
\end{proof}

\subsection{Corollaries}

Let $V$ be a smooth projective variety over $\Q$. We say that $V$ has good reduction at a prime $p$ if there exists a smooth projective model of $V$ over $\Spec(\Z_{(p)})$.  Otherwise, we say that $V$ has bad reduction at $p$.

It is well-known that the set $S$ of primes of bad reduction of $V$ is finite. Moreover, it is possible to construct a smooth projective model $\V\rightarrow \Spec(\Z_S)$.

\begin{lem}
\label{good_model}
Let $V$ be a smooth projective variety over $\Q$. Let $S$ be the set of primes of bad reduction of $V$. Then there exists a smooth projective model $\V\rightarrow \Spec(\Z_S)$ of $V$. Moreover, any such model satisfies
$$
\V(\mathcal{O}_{K,S})=V(K)
$$
for any number field $K$, and
$$
\Pic(\V)=\Pic(V).
$$
\end{lem}

\begin{proof}
The existence of a smooth projective model $\V\rightarrow \Spec(\Z_S)$ was noted above. The projectivity of $\V$ ensures us that any morphism $P:\Spec(K)\rightarrow V$ extends to a morphism $\Spec(\mathcal{O}_{K,S})\rightarrow \V$ (valuative criterion of properness).

Because $\V$ is regular, any divisor on $V$ can be extended (by scheme-theoretic closure) to a divisor on $\V$. Thus, the restriction map
$$
\Pic(\V)\rightarrow\Pic(V)
$$
is surjective. The kernel of this map is the subgroup of $\Pic(\V)$ generated by fibral (or ``vertical'') divisors. By assumption $V$ is geometrically connected, so by Zariski's connectedness theorem $\V$ has connected fibers. Therefore, $\V$ has integral fibers (because smooth and connected implies integral), so any fibral divisor is a sum of fibers. In other words, any fibral divisor is the inverse image of a divisor on $\Spec(\Z_S)$. Such a divisor is principal, because $\Z_S$ is principal. So the subgroup generated by fibral divisors is trivial, and the restriction map is bijective.
\end{proof}

\begin{cor}
\label{corollary1}
Let $V$ be a smooth projective variety over $\Q$. Then, for any $m>1$, the hypotheses of Theorem~\ref{MainTheorem} hold for $V$ and $m$ when taking $S$ to be the set of primes of bad reduction of $V$.
\end{cor}

\begin{proof}
Let $\V\rightarrow \Spec(\Z_S)$ be the model of $V$ as in Lemma~\ref{good_model}. By Kummer theory, we have a diagram with exact lines
$$
\begin{CD}
0 @>>> \gm(\V)/m @>>> H^1(\V,\mu_m) @>>> \Pic(\V)[m] @>>> 0\\
@. @VVV @VVV @| \\
0 @>>> \gm(V)/m @>>> H^1(V,\mu_m) @>>> \Pic(V)[m] @>>> 0\\
\end{CD}
$$
where the vertical lines are obtained by restriction to the generic fiber.
So, starting from $L\in \Pic(V)[m]$, we get a $\mu_m$-torsor $\mathcal{T}\rightarrow \V$ whose generic fiber $T\rightarrow V$ is a lift of $L$.

Now, if $P:\Spec(K)\rightarrow V$ is a point, we have a commutative diagram
$$
\begin{CD}
H^1(\V,\mu_m) @>P^*>> H^1(\mathcal{O}_{K,S},\mu_m) \\
@VVV @VVV \\
H^1(V,\mu_m) @>P^*>> H^1(K,\mu_m) \\
\end{CD}
$$
where the horizontal lines are pull-backs along $P$. This proves that $P^*T$ belongs to the subgroup $H^1(\mathcal{O}_{K,S},\mu_m)$.
\end{proof}

\begin{rmq}
If $A$ is an abelian variety over $\Q$, then $A(\Q)=\Pic^0(A^t)$ where $A^t$ is the dual abelian variety of $A$. It seems interesting to apply Corollary~\ref{corollary1} to this situation (i.e. taking $V=A^t$). But then one has to bound the degree of a generically finite rational map $A^t\rightarrow \mathbb{A}^l$, where $l=\dim(A)$.
\end{rmq}

Let us recall that, if $C$ is a curve defined over $\Q$, the index $\ind(C)$ of $C$ is, by definition, the quantity
$$
\ind(C) = \gcd\{[K:\Q]~|~ C(K)\neq\emptyset\}.
$$
Of course, if $C(\Q)\neq\emptyset$, then $\ind(C)=1$. The converse is false.

We now assume that $C$ is a smooth projective curve of genus $g\geq 1$. Then $C$ has a minimal regular model $\C\rightarrow \Spec(\Z)$.

Let $p$ be any prime number, and let $\Gamma_1,\dots,\Gamma_r$ be the irreducible components of the fiber of $\C$ at $p$. For each $i$, let $\delta_i$ be the geometric multiplicity of $\Gamma_i$ in the fiber of $\C$ at $p$ (see \cite[\S{}9.1, Def.~3]{NeronModels}). We let
$$
I_p(C):=\gcd(\delta_1,\dots,\delta_r).
$$

It is well-known that $I_p(C)$ divides $\ind(C)$ for all $p$, cf. \cite[Cor.~1.5]{BoschLiu}.

Let $\Jac(C)$ be the Jacobian variety of the curve $C$.  This is an abelian variety of dimension $g$. For each prime number $p$, we let $\Phi_p$ be the group of connected components of the fiber at $p$ of the N\'eron model of $\Jac(C)$, and we denote by
$$
t_p(\Jac(C)):=\# \Phi_p(\mathbb{F}_p)
$$
the order of the group of $\mathbb{F}_p$-valued points of $\Phi_p$. Usually, $t_p(\Jac(C))$ is called the Tamagawa number of $\Jac(C)$ at $p$.

\begin{propdfn}
\label{better_set}
Let $C$ be a smooth projective curve of genus $g\geq 1$ over $\Q$, and let $m>1$. We let $U(C,m)\subseteq \Spec(\Z)$ be the set of primes $p$ such that $I_p(C)=1$ and $t_p(\Jac(C))$ is coprime to $m$.
Then the following holds:
\begin{enumerate}
\item[$(i)$] If $g\geq 2$, the set $U(C,m)$ contains the set of primes of good reduction of $\Jac(C)$.
\item[$(ii)$] If $\ind(C)=1$ then $I_p(C)=1$ for all $p$. Hence, $U(C,m)$ is the set of primes $p$ such that $t_p(\Jac(C))$ is coprime to $m$.
\end{enumerate}
\end{propdfn}

\begin{proof}
If $g\geq 2$ and $\Jac(C)$ has good reduction at $p$, then it follows from \cite[Thm. 2.4]{DM} that the curve $\C$ is semi-stable at $p$, hence $I_p(C)=1$. This proves $(i)$. According to \cite[Cor.~1.5]{BoschLiu}, $I_p(C)$ divides $\ind(C)$ for all $p$, hence $(ii)$.
\end{proof}

\begin{cor}
\label{corollary2}
Let $C$ be a smooth projective curve of genus $g\geq 1$ over $\Q$, and let $m>1$. Let $S$ be the complement of $U(C,m)$ in $\Spec(\Z)$. Let $\phi:C\to \mathbb{A}^1$ be a rational map. Then (with the notation of Theorem~\ref{MainTheorem}), for all but $O(\sqrt{N})$ values $t=1,\dots,N$, we have $[\Q(P_t):\Q]=\deg \phi$ and
$$
\rk_m\Cl(\Q(P_t))\geq \rk_m \Pic(C)_{\tors}+\# S- \rk \mathcal{O}_{\Q(P_t),S}^{\times}.
$$
\end{cor}

\begin{proof}
Let $Z=U(C,m)$, and let $\pi:\C\rightarrow Z$ be the minimal regular model of $C$ over $Z$. The curve $C$ being geometrically integral, we know (see \cite[\S{}8.3, Cor.~3.6 (c)]{Liu}) that $\pi_*\mathcal{O}_{\C}=\mathcal{O}_Z$ holds.

Let $\Pic_{\C/Z}$ be the relative Picard functor of $\C$ over $Z$, and let $\Pic_{\C/Z}^0$ be its identity component. Let $\J\rightarrow Z$ be the N\'eron model of $\Jac(C)$ over $Z$, and let $\J^0$ be the identity component of $\J$. By definition of $Z$, we know that $I_p(C)=1$ for all $p\in Z$. Hence, according to \cite[\S{}9.5, Thm.~4 (b)]{NeronModels}, the functor $\Pic_{\C/Z}^0$ is representable by $\J^0$.

It is clear that $\Pic(Z)=0$. Therefore, according to \cite[\S{}8.1, Prop.~4]{NeronModels}, we have an exact sequence
$$
0 \longrightarrow \Pic(\C) \longrightarrow \Pic_{\C/Z}(Z) \longrightarrow H^2(Z,\gm).
$$
Let $\Pic^0(\C/Z)$ be the subgroup of $\Pic(\C)$ consisting of elements whose image belongs to $\Pic_{\C/Z}^0(Z)$. By looking at $m$-torsion, we get an exact sequence
$$
0 \longrightarrow \Pic^0(\C/Z)[m] \longrightarrow \Pic_{\C/Z}^0(Z)[m] \longrightarrow H^2(Z,\gm).
$$
Of course, it is possible to consider a similar exact sequence where $\C$ is replaced by $C$. Putting the two sequences together, we get a diagram with exact lines
\begin{equation}
\label{cor2diag}
\begin{CD}
0 @>>> \Pic^0(\C/Z)[m] @>>> \Pic_{\C/Z}^0(Z)[m] @>>> H^2(Z,\gm) \\
@. @VV\rho_1V @VV\rho_2V @VV\rho_3V \\
0 @>>> \Pic^0(C)[m] @>>> \Jac(C)(\Q)[m] @>>> H^2(\Spec(\Q),\gm) \\
\end{CD}
\end{equation}
where the vertical maps $\rho_i$ are obtained by restriction to the generic fiber. We claim that $\rho_2$ is an isomorphism. Indeed, $\Pic^0_{\C/Z}$ being isomorphic to $\J^0$, this map can be identified with
$$
\J^0(Z)[m]\subseteq \J(Z)[m]\simeq \Jac(C)(\Q)[m],
$$
and the inclusion is in fact an equality because $t_p(\Jac(C))$ is coprime to $m$ for all $p\in Z$. Moreover, the map $\rho_3$ is injective according to \cite[Prop. 2.1]{Gro}. It follows from the snake lemma applied to the diagram \eqref{cor2diag} that $\rho_1$ is an isomorphism.

By the same arguments as in the proof of Corollary~\ref{corollary1}, it follows that $C$, $m$ and $S$ satisfy the hypotheses of Theorem~\ref{MainTheorem}.
\end{proof}

\begin{rmq}
If $C$ has genus at least $2$, then, according to Proposition-Definition~\ref{better_set}, the set $S$ of Corollary~\ref{corollary2} is contained in the set of primes of bad reduction of $\Jac(C)$. This set is in general strictly smaller than the set of primes of bad reduction of $C$, even in the case when $C$ has a rational point. This is the main interest of Corollary~\ref{corollary2} compared to Corollary~\ref{corollary1}.
\end{rmq}

\begin{rmq}
If $C$ has genus one, then $\Jac(C)$ is an elliptic curve, that we denote by $E$. Moreover, we have
$$
\Pic^0(C)\subseteq E(\Q)\simeq \Pic^0(E),
$$
where the isomorphism follows from auto-duality of elliptic curves. Hence, if we are interested in building number fields with a large class group, we may apply Corollary~\ref{corollary2} to $E$ instead of $C$, which permits a possibly smaller $S$ (see also Corollary \ref{super}).
\end{rmq}

\begin{rmq}
By definition of $\ind(C)$, it is possible to find number fields $K_1,\ldots, K_n$ together with points $P_i\in C(K_i)$ such that the gcd of the degrees $[K_i:\Q]$ is $\ind(C)$. Hence, according to \cite[\S{}9.1, Prop.~11]{NeronModels}, the cokernel of the map $\Pic^0(C) \rightarrow \Jac(C)(\Q)$ is killed by $\ind(C)$. Therefore, if we let
$$
m'=\frac{m}{\gcd(m,\ind(C))},
$$
then we have
$$
\rk_{m'} \Pic^0(C)_{\tors} \geq \rk_m \Jac(C)(\Q)_{\tors}
$$
and equality holds when $\gcd(m,\ind(C))=1$.
\end{rmq}


\section{Applications}

\subsection{Curves}
In this section we consider some applications and further refinements of Corollary~\ref{corollary2}.  Note that if we fix the degree of $\phi$, the quantity $\rk \mathcal{O}_{\Q(P_t),S}^{\times}-\# S$ that appears in Corollary~\ref{corollary2} is minimized if every prime in $S$ is totally ramified in $\Q(P_t)$ and $\Q(P_t)$ has the maximum possible number of complex places for a number field of degree $\phi$.  In this case, we have the equalities
\begin{equation*}
\rk \mathcal{O}_{\Q(P_t),S}^{\times}-\# S=\rk \mathcal{O}_{\Q(P_t)}^{\times}=\left[\frac{\deg \phi-1}{2}\right].
\end{equation*}
By choosing appropriate maps $\phi$ in Corollary~\ref{corollary2}, we show that this advantageous situation can be achieved for hyperelliptic or, more generally, superelliptic curves.

\begin{cor}
\label{super}
Let $C$ be (a smooth projective model of) the plane curve defined by 
\begin{equation*}
y^n=f(x), \quad f(x)\in \Q[x], \quad n>1, 
\end{equation*}
with $(\deg f, n)=1$.  Let $m>1$ be an integer.  Then there exist $\gg X^{\frac{1}{(n-1)\deg f}}/\log X$ number fields $k$ of degree $[k:\Q]=n$ with discriminant $d_k$, $|d_k|<X$, and
\begin{equation*}
\rk_m \Cl(k)\geq \rk_m \Jac(C)(\Q)_{\rm{tors}}-\left[\frac{n-1}{2}\right].
\end{equation*}
\end{cor}
\begin{proof}
Let $r=\deg f$ and let $f(x)=\sum_{i=0}^r a_ix^i$.  By rescaling $x$ and $y$ we can assume that $a_i\in \Z$ for all $i$ and, using $(r,n)=1$, that $a_r=-1$.  Let $S$ be the set of primes from Corollary \ref{corollary2}. Let $M=\prod_{p\in S}p$.  Let $\phi$ be the rational function on $C$ defined by $\phi=x-\frac{1}{M}$.  For  $t\in \Z$, let $P_t=\phi^{-1}(t)$.  Then $\Q(P_t)\cong \Q\left(\sqrt[n]{f\left(\frac{tM+1}{M}\right)}\right)$.  Let $p\in S$.  Since $a_r=-1$, we have that $v_p \left(f\left(\frac{tM+1}{M}\right)\right)=-r$.  Since $(r,n)=1$, this implies that $p$ totally ramifies in $\Q(P_t)$.  Note also that $[\Q(P_t):\Q]=n$.  So every prime in $S$ is totally ramified in $\Q(P_t)$.  The condition $(r,n)=1$ also implies that $y^n=f(x)$ has a rational point at infinity.  Then Corollary~\ref{corollary2} applied to $\phi$ and $C$ implies that for all but $O(\sqrt{N})$ values $t=1,\ldots, N$, we have  $[\Q(P_t):\Q]=n$ and
\begin{equation*}
\rk_m \Cl(\Q(P_t))\geq \rk_m \Jac(C)(\Q)_{\rm{tors}}-\rk \mathcal{O}_{\Q(P_t)}^*.
\end{equation*}
Since $a_r=-1$, for $t\gg 0$, $f\left(\frac{tM+1}{M}\right)$ is negative.  It follows that for $t\gg 0$, $\Q(P_t)$ has exactly one real place if $n$ is odd and no real places if $n$ is even.  So by Dirichlet's theorem, $\rk \mathcal{O}_{\Q(P_t)}^*=\left[\frac{n-1}{2}\right]$ for $t\gg 0$.  It follows from Theorem~\ref{Zan} that there are $\gg N/\log N$ distinct number fields in the set $\{\Q(P_1),\ldots,\Q(P_N)\}$.  An easy calculation shows that $|d_{\Q(P_t)}|=O(t^{(n-1)t})$.  Combining the above statements then gives the corollary.
\end{proof}

Of particular interest is the case where $C$ is a hyperelliptic curve.

\begin{cor}
\label{chyp}
Let $C$ be a smooth projective hyperelliptic curve over $\Q$ with a rational Weierstrass point.  Let $g$ denote the genus of $C$.  Let $m>1$ be an integer.  Then there exist $\gg X^{\frac{1}{2g+1}}/\log X$ imaginary quadratic number fields $k$ with
\begin{equation*}
\rk_m \Cl(k)\geq \rk_m \Jac(C)(\Q)_{\rm{tors}},\quad |d_k|<X,
\end{equation*}
and $\gg X^{\frac{1}{2g+1}}/\log X$ real quadratic number fields $k$ with
\begin{equation*}
\rk_m \Cl(k)\geq \rk_m \Jac(C)(\Q)_{\rm{tors}}-1,\quad d_k<X.
\end{equation*}
\end{cor}
\begin{proof}
Since $C$ has a rational Weierstrass point, $C$ is birational to a plane curve defined by $y^2=f(x)$, for some $f(x)\in \Q[x]$ with $\deg f=2g+1$.  The statement on imaginary quadratic fields now follows from the proof of Corollary~\ref{super}.  The statement for real quadratic fields follows similarly from the proof of Corollary~\ref{super} using points $P_t$, $t<0$, with the extra $-1$ term coming from the rank one unit group of a real quadratic field.
\end{proof}

A classical result of Nagell \cite{Nag, Nag2} states that for any positive integer $m$ there exist infinitely many imaginary quadratic number fields with an element of order $m$ in the ideal class group.  Weinberger \cite{Wei} and Yamamoto \cite{Yam}, independently, extended Nagell's result to real quadratic fields and Yamamoto \cite{Yam} improved Nagell's result for imaginary quadratic fields to class groups of $m$-rank two.  As a sample application of Corollary~\ref{chyp}, we show that it easily implies (largely known) quantitative versions of the results of Weinberger \cite{Wei} and Yamamoto \cite{Yam}.

\begin{lem}
\label{hlem}
Let $b,c\in \Q$, with $c\neq 0$, $b^2\neq 4c^2$, and let $m>1$ be an integer.   Let $C$ be the (smooth projective model of the) hyperelliptic curve $y^2=x^{2m}+bx^m+c^2$.  Then $\rk_m \Jac(C)(\Q)_{\rm tors}\geq 2$.
\end{lem}
\begin{proof}
The curve $C$ has two $\Q$-rational points at infinity that we'll denote by $\infty$ and $\overline{\infty}$.  Let $P,\bar{P}\in C$ be the two points with coordinates $P=(0,c)$ and $\overline{P}=(0,-c)$.  Then, after possibly switching $\infty$ and $\overline{\infty}$, we have
\begin{align*}
{\rm div}(y-x^m-c)&=m(P-\infty),\\
{\rm div}(y-x^m+c)&=m(\overline{P}-\infty).
\end{align*}
Let $i,j\in \{0,\ldots, m-1\}$ with $i\geq j$, $i>0$.  Then, since $P+\overline{P}\sim \infty+\overline{\infty}$,
\begin{align*}
i(P-\infty)+j(\overline{P}-\infty)&\sim (i-j)P+j(P+\overline{P})-(i+j)\infty\\
&\sim (i-j)P+j\overline{\infty}-i\infty.
\end{align*}
Since the genus of $C$ is $m-1$ and $\infty$ is not a Weierstrass point of $C$, we have $l(i\infty)=1$ as $i\leq m-1$.  Then $i(P-\infty)+j(\overline{P}-\infty)$ is not a principal divisor and it follows that the divisors classes of $P-\infty$ and $\overline{P}-\infty$ generate a subgroup $(\Z/m)^2$ in $\Jac(C)(\Q)$.
\end{proof}

Note that there are hyperelliptic curves as in Lemma~\ref{hlem} with a rational Weierstrass point (e.g., curves with $b=-1-c^2$). Then from Corollary~\ref{chyp} we immediately obtain the following result.
\begin{cor}
\label{ri}
Let $m>1$ be an integer.  There exist $\gg X^{\frac{1}{2m-1}}/\log X$ imaginary quadratic number fields $k$ with $|d_k|<X$ and $\rk_m \Cl(k)\geq 2$ and $\gg X^{\frac{1}{2m-1}}/\log X$ real quadratic number fields $k$ with $d_k<X$ and $\rk_m \Cl(k)\geq 1$.
\end{cor}

If $m$ is odd, then Byeon \cite{Bye} and Yu \cite{Yu} have proved, for imaginary and real quadratic fields, respectively, the better lower bound of $\gg X^{1/m-\epsilon}$.  If $m$ is even, in the real quadratic case a lower bound of $\gg X^{1/m}$ was proved by Chakraborty, Luca, and
Mukhopadhyay \cite{CLM}.  The imaginary quadratic case of Corollary~\ref{ri} with $m$ even appears to be new.

Using the results presented here, it is in fact possible to derive quantitative versions of many of the known results on constructing number fields with class groups of large rank.  This aspect of the results will be pursued in a future paper.

Corollary~\ref{chyp} raises the following natural question.

\begin{question}
\label{Q1}
Let $p$ be an odd prime.  Do there exist hyperelliptic curves $C$ over $\Q$ with $\rk_p \Jac(C)(\Q)_{\rm{tors}}$ arbitrarily large?
\end{question}

More generally, one can ask whether there exist curves $C$ over $\Q$ of gonality $n$ with $\rk_p \Jac(C)(\Q)_{\rm{tors}}$ arbitrarily large, where $p$ is a prime not dividing $n$ (in the case $p|n$, a positive answer is easily obtained from curves of the form $y^n=f(x)$, where $f(x)$ splits over $\Q$).  This question does not appear to have been extensively investigated.  Previous papers studying the problem of constructing Jacobians of curves with large rational torsion subgroups have primarily focused on either curves of low genus (e.g, \cite{Poo}), or on producing a rational torsion point of large order in the Jacobian of a curve of genus $g$, for every genus $g$ (e.g., \cite{Fly}).

\subsection{Torsion line bundles on arithmetic varieties}

We consider here the question formulated in \cite{bisipappas}, namely:

\begin{question}
\label{APquest}
Let $\mathcal{L}$ be a nontrivial line bundle over an arithmetic variety $\X$. Is it possible to find a section $P$ of $\X$ over some number field such that the pull-back of $\mathcal{L}$ by this section is also nontrivial ?
\end{question}

Here ``arithmetic variety'' means a normal scheme whose structural morphism to $\Spec(\Z)$ is proper and flat. Let us note that, under these assumptions, the generic fiber of $\X$ need not be geometrically integral.

\begin{cor}
\label{arithmetic_bundles}
Let $\X$ be an arithmetic variety whose generic fiber $X$ is smooth and geometrically integral. Let $S$ be the set of places of bad reduction of $\X$, and let $\X_S:=\X\times_{\Z}\Spec(\Z_S)$. Then, given an integer $m>1$, there exists an infinity of number fields $K$ with a point $P\in X(K)$ such that the image of the restriction map
$$
P^*:\Pic(\X_S)[m]\longrightarrow \Cl(\mathcal{O}_{K,S})[m]
$$
satisfies
$$
\rk_m \img(P^*)\geq \rk_m \Pic(\X_S)[m] +\#S - \rk \mathcal{O}_{K,S}^{\times}.
$$
If $X$ is the normalization of a plane curve defined by
$$
y^n=f(x)
$$
with $(\deg(f),n)=1$, then the same holds with
$$
\rk_m \img(P^*)\geq  \rk_m \Pic(\X_S)[m]- \left[\frac{n-1}{2}\right].
$$
\end{cor}

\begin{proof}
Let us note here that our $S$ is \emph{a priori} larger than the set of places of bad reduction of $X$, but the arguments in the proof of Lemma~\ref{good_model} still hold, and give us the equality $\Pic(\X_S)=\Pic(X)$.
Then Corollary~\ref{corollary1} can be applied with the same set $S$, and the first statement can be extracted from the proof of Theorem~\ref{MainTheorem}. The second statement follows from the proof of Corollary~\ref{super}.
\end{proof}

\begin{cor}
\label{partial_answer}
Let $\X$ be an arithmetic variety whose generic fiber is a smooth hyperelliptic curve with a rational Weierstrass point. Let $\mathcal{L}$ be a torsion line bundle over $\X$ whose generic fiber is nontrivial. Then there exists an infinity of imaginary quadratic fields $K$ with a section $P$ of $\X$ over $\mathcal{O}_K$ such that the pull-back $P^*\mathcal{L}$ is nontrivial.
\end{cor}

\begin{proof}
Obviously, it suffices to prove the result when $\mathcal{L}$ has prime order $p$. In this case, the equality $\rk_p \img(P^*)=\rk_p \Pic(\X_S)[p]$ implies that $\ker(P^*)=0$. Therefore, the second statement in Corollary~\ref{arithmetic_bundles} implies that, for an infinity of imaginary quadratic fields $K$, the map
$$
P^*:\Pic(\X_S)[p]\longrightarrow \Cl(\mathcal{O}_{K,S})[p]
$$
is injective. On the other hand, we know that $\mathcal{L}$ has a nontrivial generic fiber, so the restriction of  $\mathcal{L}$ to $\X_S$ is also nontrivial. This proves the result.
\end{proof}


\bibliographystyle{amsalpha}
\bibliography{class2.bib}

\providecommand{\bysame}{\leavevmode\hbox to3em{\hrulefill}\thinspace}
\providecommand{\MR}{\relax\ifhmode\unskip\space\fi MR }
\providecommand{\MRhref}[2]{%
  \href{http://www.ams.org/mathscinet-getitem?mr=#1}{#2}
}
\providecommand{\href}[2]{#2}
\begin{thebibliography}{Mes83b}

\bibitem[AP00]{bisipappas}
A.~Agboola and G.~Pappas, \emph{Line bundles, rational points and ideal
  classes}, Math. Res. Lett. \textbf{7} (2000), no.~5-6, 709--717.

\bibitem[BL99]{BoschLiu}
S.~Bosch and Q.~Liu, \emph{Rational points of the group of components of a
  {N}\'eron model}, Manuscripta Math. \textbf{98} (1999), no.~3, 275--293.

\bibitem[BLR90]{NeronModels}
S.~Bosch, W.~L{\"u}tkebohmert, and M.~Raynaud, \emph{N\'eron models},
  Ergebnisse der Mathematik und ihrer Grenzgebiete (3), vol.~21,
  Springer-Verlag, Berlin, 1990.

\bibitem[Bye06]{Bye}
D.~Byeon, \emph{Imaginary quadratic fields with noncyclic ideal class groups},
  Ramanujan J. \textbf{11} (2006), no.~2, 159--163.

\bibitem[CLM08]{CLM}
K.~Chakraborty, F.~Luca, and A.~Mukhopadhyay, \emph{Exponents of class groups
  of real quadratic fields}, Int. J. Number Theory \textbf{4} (2008), no.~4,
  597--611.

\bibitem[Coh81]{Coh2}
S.~D. Cohen, \emph{The distribution of {G}alois groups and {H}ilbert's
  irreducibility theorem}, Proc. London Math. Soc. (3) \textbf{43} (1981),
  no.~2, 227--250.

\bibitem[DM69]{DM}
P.~Deligne and D.~Mumford, \emph{The irreducibility of the space of curves of
  given genus}, Inst. Hautes \'Etudes Sci. Publ. Math. (1969), no.~36, 75--109.

\bibitem[DZ94]{Zan}
R.~Dvornicich and U.~Zannier, \emph{Fields containing values of algebraic
  functions}, Ann. Scuola Norm. Sup. Pisa Cl. Sci. (4) \textbf{21} (1994),
  no.~3, 421--443.

\bibitem[DZ95]{Zan2}
\bysame, \emph{Fields containing values of algebraic functions. {II}. ({O}n a
  conjecture of {S}chinzel)}, Acta Arith. \textbf{72} (1995), no.~3, 201--210.

\bibitem[Fly91]{Fly}
E.~V. Flynn, \emph{Sequences of rational torsions on abelian varieties},
  Invent. Math. \textbf{106} (1991), no.~2, 433--442.

\bibitem[Gro68]{Gro}
A.~Grothendieck, \emph{Le groupe de {B}rauer. {III}. {E}xemples et
  compl\'ements}, Dix {E}xpos\'es sur la {C}ohomologie des {S}ch\'emas,
  North-Holland, Amsterdam, 1968, pp.~88--188.

\bibitem[HLP00]{Poo}
E.~W. Howe, F.~Lepr{\'e}vost, and B.~Poonen, \emph{Large torsion subgroups of
  split {J}acobians of curves of genus two or three}, Forum Math. \textbf{12}
  (2000), no.~3, 315--364.

\bibitem[Lev07]{Levin}
A.~Levin, \emph{Ideal class groups, {H}ilbert's irreducibility theorem, and
  integral points of bounded degree on curves}, J. Th\'eor. Nombres Bordeaux
  \textbf{19} (2007), no.~2, 485--499.

\bibitem[Liu02]{Liu}
Q.~Liu, \emph{Algebraic geometry and arithmetic curves}, Oxford Graduate Texts
  in Mathematics, vol.~6, Oxford University Press, Oxford, 2002, Translated
  from the French by Reinie Ern{\'e}, Oxford Science Publications.

\bibitem[Mes80]{Mes}
J.-F. Mestre, \emph{Courbes elliptiques et groupes de classes d'id\'eaux de
  certains corps quadratiques}, Seminar on Number Theory, 1979--1980 (French),
  Univ. Bordeaux I, Talence, 1980, pp.~Exp. No. 15, 18.

\bibitem[Mes83a]{Mes3}
\bysame, \emph{Courbes elliptiques et groupes de classes d'id\'eaux de certains
  corps quadratiques}, J. Reine Angew. Math. \textbf{343} (1983), 23--35.

\bibitem[Mes83b]{Mes2}
\bysame, \emph{Groupes de classes d'ideaux non cycliques de corps de nombres},
  Seminar on number theory, Paris 1981--82 (Paris, 1981/1982), Progr. Math.,
  vol.~38, Birkh\"auser Boston, Boston, MA, 1983, pp.~189--200.

\bibitem[Mes92]{Mes4}
\bysame, \emph{Corps quadratiques dont le {$5$}-rang du groupe des classes est
  {$\geq 3$}}, C. R. Acad. Sci. Paris S\'er. I Math. \textbf{315} (1992),
  no.~4, 371--374.

\bibitem[Nag22]{Nag}
T.~Nagell, \emph{\"{U}ber die {K}lassenzahl imagin\"ar-quadratischer
  {Z}ahlk\"orper}, Abh. Math. Sem. Univ. Hamburg \textbf{1} (1922), 140--150.

\bibitem[Nag02]{Nag2}
\bysame, \emph{Collected papers of {T}rygve {N}agell. {V}ol. 1}, Queen's Papers
  in Pure and Applied Mathematics, vol. 121, Queen's University, Kingston, ON,
  2002.

\bibitem[Ser97]{Ser}
J.-P. Serre, \emph{Lectures on the {M}ordell-{W}eil theorem}, third ed.,
  Aspects of Mathematics, Friedr. Vieweg \& Sohn, Braunschweig, 1997.

\bibitem[Wei73]{Wei}
P.~J. Weinberger, \emph{Real quadratic fields with class numbers divisible by
  {$n$}}, J. Number Theory \textbf{5} (1973), no.~3, 237--241.

\bibitem[Yam70]{Yam}
Y.~Yamamoto, \emph{On unramified {G}alois extensions of quadratic number
  fields}, Osaka J. Math. \textbf{7} (1970), 57--76.

\bibitem[Yu02]{Yu}
G.~Yu, \emph{A note on the divisibility of class numbers of real quadratic
  fields}, J. Number Theory \textbf{97} (2002), no.~1, 35--44.

\end{thebibliography}


\vskip 20pt

Jean Gillibert
\smallskip

Institut de Math\'ematiques de Bordeaux

CNRS UMR 5251

Universit\'e Bordeaux 1

351, cours de la Lib\'eration

F-33400 Talence, France.
\medskip

\texttt{jean.gillibert@math.u-bordeaux1.fr}

\bigskip

Aaron Levin
\smallskip

Department of Mathematics

Michigan State University

619 Red Cedar Road

East Lansing, MI 48824
\medskip

\texttt{adlevin@math.msu.edu}

\end{document}